\pdfoutput=1

\documentclass[12pt]{amsart}
\usepackage[utf8]{inputenc}
\usepackage{amsmath}
\usepackage{amsthm}
\usepackage{amscd}
\usepackage{comment}
\usepackage[english]{babel}
\usepackage{tikz}
\usepackage{color,graphicx,amssymb,tikz-cd,wrapfig}
\usepackage{amsmath,amsthm,amsfonts,amssymb,stmaryrd}
\usepackage{bm}
\usepackage[rightcaption]{sidecap}
\usepackage[margin=20truemm]{geometry}

\usepackage{enumerate}
\usepackage{pb-diagram}
\usepackage{graphicx} %package to manage images
\usepackage[all]{xy}
\usepackage{appendix}

\title{An Operator-Valued Haagerup Inequality for Hyperbolic Groups}

\author{Ryo Toyota}

\address{Ryo Toyota: Department of Mathematics, Texas A\&M University, TX, USA}

\email{ryo-toyota@tamu.edu}

\author{Zhiyuan Yang}

\address{Zhiyuan Yang: Department of Mathematics, Texas A\&M University, TX, USA}

\email{zhiyuanyang@tamu.edu}

\begin{document}

\newtheorem{thm}{Theorem}
\newtheorem{dfn}[thm]{Definition}
\newtheorem{lem}[thm]{Lemma}
\newtheorem{cor}[thm]{Corollary}
\newtheorem{rmk}[thm]{Remark}
\newtheorem*{pf}{Proof}
\newtheorem{ex}{Exercise}
\newtheorem{prop}[thm]{Proposition}
\newtheorem*{claim}{Claim}

\newcommand{\CC}{ {\mathbb C} }
\newcommand{\HH}{{\mathcal H}}
\newcommand{\FF}{{\mathbb F}}

\maketitle

\begin{abstract}
    We study an operator-valued generalization of the Haagerup inequality for Gromov hyperbolic groups. In 1978, U. Haagerup showed that if $f \in \CC[\FF_r]$ is supported on the $k$-sphere $S_k=\{x\in \FF_r:\ell(x)=k\}$, then we have
    $\left \| \sum_{x \in S_k} f(x)\lambda(x)\right \|_{B(\ell^2(\FF_r))} \leq (k+1)\|f\|_2.$
   An operator-valued generalization of it was initiated by U. Haagerup and G. Pisier. One of the most complete form was given by A. Buchholz, where the $\ell^2$-norm in the original inequality was replaced by $k+1$ different matrix norms associated to word decompositions (this type of inequality is also called Khintchine-type inequality). We provide a generalization of Buchholz's result for hyperbolic groups.
\end{abstract}

\section{Introduction}
In this paper, we study an operator-valued generalization of the Haagerup inequality for Gromov hyperbolic groups. For a given finitely generated group $G$, we denote the left regular representation of its group algebra $\CC[G]$ to $\ell^2(G)$ by $\lambda$. In Lemma 1.4 of \cite{haagerup1978example}, Haagerup showed that for free groups, the operator norm of the left regular representation, which is difficult to compute in general is dominated by a certain $\ell^2$-norm $\|\cdot\|_2$.

\begin{lem}[\cite{haagerup1978example} Lemma 1.4]
\rm{
    Let $\FF_r$ be the free group with $r$-generators with the canonical length function $\ell$. If $f$ is a complex-valued function on $\FF_r$ supported on the $k$-sphere $S_k:=\{x \in \FF_r: \ell(x)=k\}$, then we have
    \begin{equation}\label{Haagerup}
        \left \| \sum_{x \in S_k} f(x)\lambda(x)\right \| \leq (k+1)\|f\|_2.
    \end{equation}
    }
\end{lem}

This inequality implies the rapid decay property of free groups, namely for all $f \in \CC[G]$, we have $\|\lambda(f)\|\leq 2 \left ( \sum_{x \in F_r}|f(x)|^2(1+\ell(x))^4 \right )^{\frac{1}{2}}$. The same inequality also holds for hyperbolic groups up to a constant factor. (See Proposition 3.3 and Proposition 4.3 of \cite{ozawa2005hyperbolic}.)
We investigate the case where a function $f$ on $G$ takes operator values, namely we consider the tensor product $\CC[G]\otimes B(\HH)$, where $\HH$ is a Hilbert space and $B(\HH)$ is the set of all bounded linear operators on $\HH$. This direction of generalization was first initiated by Haagerup and Pisier in \cite{haagerup1993bounded} and they showed the following inequality.
\begin{prop}[\cite{haagerup1993bounded} Proposition 1.1]
\rm{
    If $f\in \CC[\FF_r]\otimes B(\HH)$ is supported on $S_1$, then we have
    \begin{align*}
        \left \| \sum_{x \in S_1} \lambda(x)\otimes f(x) \right \|_{B(\ell^2(G)\otimes \HH)} \leq 2 \max \left\{ \left\| \sum_{x\in S_1} f(x)^*f(x) \right\|_{B(\HH)}^{\frac{1}{2}},\left\|\sum_{x \in S_1}f(x)f(x)^* \right \|_{B(\HH)}^{\frac{1}{2}} \right\}.
    \end{align*}
    }
\end{prop}

The generalization for a function supported on the $k$-sphere $S_k$ for a general positive integer $k$ was studied by Buchholz in \cite{buchholz1999norm}. He replaced the term $(k+1)\|f\|_2$ of \eqref{Haagerup} by the sum of $k+1$ different matrix norms. In order to state his inequality we introduce the following notations.

\begin{dfn}
\rm{
    Let $G$ be a finitely generated group with a symmetric word length $\ell$ (namely $
    \ell(x) = \ell(x^{-1})$). For each positive integer $k$, let $\CC[G]_k$ be the set of all scalar valued functions supported on the $k$-sphere $S_k := \{ g\in G: \ell(g)= k \}$, and let also $ \CC[G]_{\leq k} $ be the set of all scalar valued function supported on the $k$-ball $B_k:= \{ g\in G: \ell(g)\leq k \}$. For $f \in \CC[G]_{\leq k}\otimes B(\HH)$, and two integers $i,j\geq 0$,
    we define a $B(\HH)$-entries $S_i\times S_{j}$-matrix $M_{i,j}(f)$ by
    \begin{align*}
        M_{i,j}(f):=(f(y_1y_2^{-1}))_{y_1 \in S_{i},y_2\in S_{j}}:\HH^{S_{j}} \rightarrow \HH^{S_{i}}.
    \end{align*}
    }
    Equivalently, we can also consider $ M_{i,j}(f) $ as a restriction of $(\lambda\otimes 1)(f)\in B(\ell^2(G)\otimes \HH)$ from $\ell^2(S_j)\otimes \HH$ to $\ell^2(S_i)\otimes \HH$.
\end{dfn}

For free group, Buchholz \cite{buchholz1999norm} proved the following inequality (see Theorem 9.7.4 of \cite{pisier2003introduction} for another reference).
\begin{thm}[\cite{buchholz1999norm} Theorem 2.8 and its proof]
\rm{
    Let $G = \mathbb{F}_m$ be a finitely generated free group and fix a positive integer $k$. For any $f \in \CC[G]_{ k}\otimes B(\HH)$, we have
    \begin{equation}\label{free strong Haagerup}
        \|(\lambda\otimes 1)(f)\|_{B(\ell^2(G))\otimes \HH}\leq \sum_{j=0}^k \|M_{j,k-j}(f)\| \leq (k+1)\max_{j=0,1,\cdots ,k}\|M_{j,k-j}(f)\|.
    \end{equation}
    }
\end{thm}

In particular, if $G = \FF_m$, $\HH = \CC$, and $f \in \CC[G]_{ k}$, then since the operator norm $\|M_{j,k-j}(f)\|$ is dominated by its Hilbert-Schmidt norm 
\begin{align*}
    \|M_{j,k-j}(f)\|_{HS}=\left(\sum_{y\in S_k} |f(y)|^2\right)^{1/2} =\|f\|_2, 
\end{align*}
the above result is stronger than the original Haagerup inequality.

This type of inequality is also called Khintchine-type inequality, and we refer the reader to \cite{ricard2006khintchine} for a generalization for reduced (amalgamated) free products. Recently, it was also shown that similar operator-valued Haagerup (Khintchine-type) inequality holds for deformations of group algebras of right-angled Coxeter groups in \cite{caspers2021graph}.

We generalize the inequality \eqref{free strong Haagerup} to Gromov word hyperbolic groups and $f\in \CC[G]_{\leq k} \otimes B(\HH)$ supported on the ball instead of the sphere. First, we recall a definition of hyperbolic groups following \cite{ozawa2005hyperbolic}, which is convenient for our purpose. 

\begin{dfn}\label{hyperbolic defn}
\rm{
    Let $(X,d)$ be a metric space and $\delta\geq 0$ be a constant. We say that $(X,d)$ is $\delta$-hyperbolic if for any four points $x,y,z,w \in X$, we have
    \begin{equation}\label{square inequality}
        d(x,y)+d(z,w) \leq \max\{d(x,z)+d(y,w),d(x,w)+d(y,z)\}+\delta.
    \end{equation}
    }
\end{dfn}

\begin{dfn}
\rm{
    Let $G$ be a finitely generated group with a symmetric word length $\ell$. The right invariant distance induced by $\ell$ is denoted by $d$ (i.e. $d(x,y)=\ell(xy^{-1})$). For an integer $\delta$, we say that $G$ is $\delta$-hyperbolic if the metric space $(G,d)$ is $\delta$-hyperbolic in the sense of Definition \ref{hyperbolic defn}. We say that $G$ is hyperbolic if it is $\delta$-hyperbolic for some $\delta \geq 0$.
    }
\end{dfn}

Here we can state our main theorem. Although related results are likely known to experts, we are not aware of any reference in the literature where this operator-valued extension is explicitly formulated and proved for hyperbolic groups.

\begin{thm}\label{main theorem}
\rm{
    Assume $G$ is a $\delta$-hyperbolic group and we fix a positive integer $k$. For any $f \in \CC[G]_{\leq k}\otimes B(\HH)$, we have
    \begin{align*}
        \|(\lambda \otimes 1)(f)\|&\leq 2 \cdot \#B_{2\delta}\cdot \sum_{\substack{i,j\geq 0\\ k\leq i+j\leq k+\delta +1}}\|M_{i,j}(f)\|\\
        &\leq (\delta+2)\cdot \#B_{2\delta}\cdot(2k+\delta+3) \cdot \max_{\substack{i,j\geq 0\\k\leq i+j \leq k+\delta+1}} \|M_{i,j}(f)\|,
    \end{align*}
    where $\#B_{s}$ is the cardinality of the ball $ B_{s}:= \{x\in G: \ell(x)\leq s \} $ with $s \geq 0$.
    }
\end{thm}

In the next section, we prove the following key lemma. For scalar valued cases, this is called the Haagerup type condition in \cite{ozawa2005hyperbolic} and can be used to obtain a compact quantum metric structure (in the sense of M. Rieffel \cite{rieffel2004gromov}) on $\CC[G]$ for hyperbolic groups $G$.

\begin{lem}[Operator valued Haagerup type condition]\label{H-type inequality}
\rm{
For each positive integer $m$, the orthogonal projection onto the space $\CC[G]_{m}\subset \ell^2(G)$ is denoted by $P_m\in B(\ell^2(G))$.
    If $G$ is a $\delta$-hyperbolic group, then for any positive integers $k,m,n$ with $|m-n|\leq k$ and any $f \in \CC[G]_{\leq k}\otimes B(\HH)$, we have
    \begin{align*}
        \|(P_m\otimes 1)(\lambda \otimes 1)(f)(P_n\otimes 1)\| \leq \#B_{2\delta}\cdot\sum_{s=0}^{\delta}\|M_{k-\lfloor \frac{p}{2}\rfloor ,\lceil \frac{p}{2}\rceil+s}(f)\|,
    \end{align*}
    where $ p = n+k-m $.
    }
\end{lem}

We conclude this section by proving our main theorem Theorem \ref{main theorem} using Lemma \ref{H-type inequality} and will prove Lemma \ref{H-type inequality} in the next section.
\begin{proof}[Proof of Theorem \ref{main theorem}]
    By Lemma \ref{H-type inequality}, we have
    \begin{align*}
        \|(\lambda \otimes 1)(f)\|&=\left \|\sum_{r=-k}^k \sum_{m=r}^{\infty} (P_m\otimes 1)(\lambda \otimes 1)(f)(P_{m-r}\otimes 1) \right \|\\
        & \leq \sum_{r=-k}^k \left \|\sum_{m=r}^{\infty} (P_m\otimes 1)(\lambda \otimes 1)(f)(P_{m-r}\otimes 1)\right \|\\
        & = \sum_{r=-k}^k \sup_{m\geq r} \{\|(P_m\otimes 1)(\lambda \otimes 1)(f)(P_{m-r}\otimes 1)\| \}\\
        & \leq \#B_{2\delta}\cdot \sum_{r=-k}^k \sum_{s=0}^{\delta} \|M_{k-\lfloor \frac{k-r}{2}\rfloor ,\lceil \frac{k-r}{2}\rceil+s}(f)\| \\
        & \leq 2\cdot \#B_{2\delta} \sum_{\substack{ i,j\geq 0\\ k\leq i+j\leq k+\delta +1}} \|M_{i,j}(f)\|.
    \end{align*}
\end{proof}

\section{Proof of Lemma \ref{H-type inequality}}

Finally, we prove Lemma \ref{H-type inequality}. Our proof is inspired by \cite{ozawa2005hyperbolic} Section 4. Take any $\xi \in \CC[G]_{n}\otimes \HH$ and $\eta \in \CC[G]_{m}\otimes\HH$.
\begin{equation}
\begin{split}
    \langle \eta, (\lambda\otimes 1)(f)\xi\rangle_{\ell^2(G)\otimes \HH}
    \label{inner product} =\sum_{x} \left \langle \eta(x), \sum_{\substack{y,z\\yz=x}} f(y)\xi(z) \right \rangle_{\HH}
    \end{split}
\end{equation}

Let $p:=n+k-m$. For every $x\in S_m$, we choose a decomposition $x= x_1x_2$ with $x_1\in S_{k-\lfloor \frac{p}{2}\rfloor}$ and $x_2\in S_{n-\lceil \frac{p}{2}\rceil}$. We denote this choice by the map $\varphi: (x,p)\mapsto (x_1,x_2)$. For $y,z$ such that $yz=x$, $\ell(y)\leq k$ and $\ell(z)=n$, denote $ u = y^{-1}x_1 $, then we have $z= ux_2$ as in the following picture.

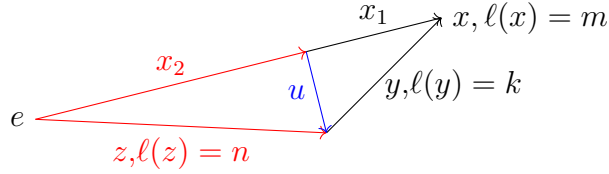
\begin{center}
        \begin{tikzpicture}[scale=0.9]
            \draw[red][->](0,0)--(4,1);
            \draw[->](4,1)--(6,1.5);
            \draw[red][->](0,0)--(4.3,-0.2);
            \draw[->](4.3,-0.2)--(6,1.5);
            \draw[blue][->](4,1)--(4.3,-0.2);
            \draw(0,0)node[left]{$e$};
            \draw(6,1.5)node[right]{$x,\ell(x)=m$};
            \draw[red](2.15,-0.1)node[below]{$z$,$\ell(z)=n$};
            \draw(5.0,0.5)node[right]{$y$,$\ell(y) \leq k$};
            \draw[red](2,0.5)node[above]{$x_2$};
            \draw(5,1.25)node[above]{$x_1$};
            \draw(4.15,0.4)[blue]node[left]{$u$};
    \end{tikzpicture}
\end{center}

By applying \eqref{square inequality} to $e,z,x,x_2$, we have
\begin{align*}
   \ell(u)+m= d(z,x_2)+d(e,x)&\leq \max\{d(x,x_2)+\ell(z),d(e,x_2)+\ell(y)\} +\delta\\
  &=\max\{\ell(x_1)+\ell(z),\ell(x_2)+\ell(y)\}+\delta\\
    &= k+n -\lfloor \frac{p}{2}\rfloor +\delta.
\end{align*}
Combining this with the triangle inequality for $y=x_1u^{-1}$, we have
\begin{equation}\label{length of u}
    \lceil\frac{p}{2}\rceil =\ell(z)-\ell(x_2) \leq \ell(u) \leq (n+k-m)-\lfloor \frac{p}{2}\rfloor +\delta \leq \lceil\frac{p}{2} \rceil+\delta.
\end{equation}

We can rewrite \eqref{inner product} as
\begin{align}
    &\langle \eta, (\lambda\otimes 1)(f)\xi\rangle_{\ell^2(G)\otimes \HH}\nonumber \\=& \sum_{x_2\in S_{n-\lceil \frac{p}{2}\rceil}} \sum_{ \substack{ x_1\in S_{k-\lfloor \frac{p}{2}\rfloor}\\\varphi(x_1x_2,p)=(x_1,x_2) } }\sum_{u\in G}\langle \eta(x_1x_2),f(x_1u^{-1})\xi(ux_2) \rangle\nonumber
  \\=& \sum_{x_2\in S_{n-\lceil \frac{p}{2}\rceil}} \sum_{ x_1\in S_{k-\lfloor \frac{p}{2}\rfloor} }\sum_{u: \lceil\frac{p}{2} \rceil\leq \ell(u)\leq \lceil \frac{p}{2}\rceil+\delta}\langle \delta_{ \varphi(x_1x_2,p),(x_1,x_2)}\eta(x_1x_2),f(x_1u^{-1})\xi(ux_2) \rangle\nonumber \\
  = &\sum_{x_2\in S_{n-\lceil \frac{p}{2}\rceil}} \sum_{ x_1\in S_{k-\lfloor \frac{p}{2}\rfloor} }\sum_{s=0}^{\delta}\sum_{u\in S_{\lceil \frac{p}{2}\rceil +s}}\langle \delta_{ \varphi(x_1x_2,p),(x_1,x_2)}\eta(x_1x_2),f(x_1u^{-1})\xi(ux_2) \rangle,\label{inner product 2}
\end{align}
where $ \delta_{ \varphi(x_1x_2,p),(x_1,x_2)} $ equals $1$ when $ \varphi(x_1x_2,p)=(x_1,x_2) $ and equals $0$ otherwise.
For each $x_2\in S_{n-\lceil \frac{p}{2}\rceil}$ and $0\leq s\leq \delta$, we define vectors $\eta_{x_2} \in \HH^{S_{k-\lfloor \frac{p}{2}\rfloor}}$ and $\xi_{x_2,s} \in \HH^{S_{\lceil \frac{p}{2}\rceil +s} }$ by
$\eta_{x_2}(x_1):=\delta_{ \varphi(x_1x_2,p),(x_1,x_2)}\eta(x_1x_2)\in \HH$ and $\xi_{x_2,s}(u):=\xi(ux_2)\in \HH$ for $x_1 \in S_{k-\lfloor \frac{p}{2}\rfloor}  $ and $u\in S_{\lceil \frac{p}{2}\rceil +s} $. Then \eqref{inner product 2} is equal to
\begin{align*}
    \sum_{x_2\in S_{n-\lceil \frac{p}{2}\rceil}}\sum_{ s=0}^{\delta} \langle \eta_{x_2}, M_{k-\lfloor \frac{p}{2}\rfloor,\lceil \frac{p}{2}\rceil+s}(f) \xi_{x_2,s}\rangle_{\HH^{S_{k-\lfloor \frac{p}{2}\rfloor}}}.
\end{align*}
Therefore, we have by the triangle inequality and the Cauchy-Schwarz inequality
\begin{equation}\label{vector norm estimate}
\begin{split}
   & |  \langle \eta, (\lambda\otimes 1)(f)\xi\rangle_{\ell^2(G)\otimes \HH} | \\
    \leq &\sum_{x_2\in S_{n-\lceil \frac{p}{2}\rceil}}\sum_{s=0}^{\delta}  \|M_{k-\lfloor \frac{p}{2}\rfloor,\lceil \frac{p}{2}\rceil+s}(f)\| \cdot\| \eta_{x_2}\| \cdot\| \xi_{x_2,s}\| \\
     \leq &\sum_{s=0}^{\delta}  \| M_{k-\lfloor \frac{p}{2}\rfloor,\lceil \frac{p}{2}\rceil+s}(f)\| (\sum_{x_2\in S_{n-\lceil \frac{p}{2}\rceil}} \|\eta_{x_2}\|^2)^{\frac{1}{2}} ( \sum_{x_2\in S_{n-\lceil \frac{p}{2}\rceil}}\|\xi_{x_2,s}\|^2)^{\frac{1}{2}}
    \end{split}
\end{equation}
Now we compare $ (\sum_{x_2\in S_{n-\lceil \frac{p}{2}\rceil}}\|\eta_{x_2}\|^2)^{\frac{1}{2}}$ and $\|\eta\|$.
For each $x \in S_m$, we count how many times $ \eta(x) $ appears in the sum
\begin{equation}\label{eta_v}
\begin{split}
    \sum_{x_2\in S_{n-\lceil \frac{p}{2}\rceil}}\|\eta_{x_2}\|^2=\sum _{x_2 \in S_{n-\lceil\frac{p}{2}\rceil}} \sum_{x_1\in S_{k-\lfloor \frac{p}{2}\rfloor}}\|\delta_{ \varphi(x_1x_2,p),(x_1,x_2)}\eta(x_1x_2)\|^2\\ \leq\sum _{x_2 \in S_{n-\lceil \frac{p}{2}\rceil}} \sum_{x_1\in S_{k-\lfloor \frac{p}{2}\rfloor}}\|\eta(x_1x_2)\|^2
    \end{split}
\end{equation}
If there are $x_1,x_1'\in S_{k-\lfloor \frac{p}{2}\rfloor}$ and $x_2,x_2'\in S_{n-\lceil \frac{p}{2}\rceil}$ such that $x_1x_2 = x =x_1'x_2'$, then by applying \eqref{square inequality} for $e,x_2 ,x ,x_2'$ we have
\begin{align*}
    d(x_2,x_2')\leq (k-\lfloor \frac{p}{2}\rfloor+n-\lceil \frac{p}{2}\rceil-m)+\delta=\delta.
\end{align*}
Therefore, for each $x \in S_m$
\begin{align*}
    \#\{(x_1,x_2) \in S_{k-\lfloor \frac{p}{2}\rfloor}\times S_{n-\lceil \frac{p}{2}\rceil}:x_1x_2 = x\} \leq \#B_{\delta}
\end{align*}
So by \eqref{eta_v}, we have $ \sum_{x_2\in S_{n-\lceil \frac{p}{2}\rceil}}\|\eta_{x_2}\|^2 \leq \#B_{\delta}\cdot\|\eta\|^2$. Similarly, for each fixed $s$, we can bound $ (\sum_{x_2\in S_{n-\lceil \frac{p}{2}\rceil}}\|\xi_{x_2,s}\|^2)^{1/2} $ by $ \|\xi\| $: 

For fixed $z\in S_n$,
if $z=u'x_2'$ for another pair $(x_2',u')\in S_{n-\lceil \frac{p}{2}\rceil}\times S_{\lceil \frac{p}{2}\rceil+s}$, we have similarly $d(x_2,x_2')\leq \delta+s$.
Therefore, $$ \#\{  (x_2,u)\in S_{n-\lceil \frac{p}{2}\rceil}\times S_{\lceil \frac{p}{2}\rceil+s}: ux_2 = z \}\leq \#B_{\delta +s }\leq \#B_{2\delta }.  $$
Hence $$ \sum_{x_2\in S_{n-\lceil \frac{p}{2}\rceil}}\|\xi_{x_2,s}\|^2= \sum_{x_2\in S_{n-\lceil \frac{p}{2}\rceil}}\sum_{u\in S_{\lceil \frac{p}{2}\rceil+s}}\|\xi(ux_2)\|^2\leq \# B_{2\delta} \|\xi\|^2. $$
Applying these to \eqref{vector norm estimate}, we obtain the desired result. \qed

\section{Some Remarks}

\begin{rmk}
\rm{
    One can give a direct proof for the exactness of hyperbolic groups using Theorem \ref{main theorem}. (Of course, the exactness is well known and the proof can be found at Section 5.3 of \cite{brown2008textrm}.) The same proof is used to show the exactness of the reduced free products of exact $C^*$-algebras in Theorem 4.1 of \cite{ricard2006khintchine}. Take any $C^*$-algebra $B$ with a closed ideal $I$. We denote two quotient maps by
    \begin{align*}
        \rho:B\twoheadrightarrow B/I \text{ and } \Tilde{\rho} :C^*_r(G)\otimes_{\text{min}} B \twoheadrightarrow (C^*_r(G)\otimes_{\text{min}} B)/(C^*_r(G)\otimes_{\text{min}} I ).
    \end{align*}
    $ G $ is exact if and only if $\|(\rho\otimes Id)(f)\|_{\text{min}}\geq\|\Tilde{\rho}(f)\|$ for any $B$ and $f \in \mathbb{C}[G]\otimes B$. Note that Theorem \ref{main theorem} states that there is an (possibly non-isometric) embedding $\iota$ into some large matrix algebra $M_N$ (, which is nuclear):
    \begin{align*}
        \iota=\bigoplus_{k\leq i+j \leq k+\delta}M_{i,j}(\cdot):\CC[G]_{\leq k}\hookrightarrow \bigoplus_{k\leq i+j \leq k+\delta}M_{S_i,S_j}\subset M_N
    \end{align*}
    such that $\|\iota\|_{cb}\leq 1$ and $\|\iota^{-1}\|_{cb} \leq (\delta+2)\cdot \#B_{2\delta}\cdot(2k+\delta+3) \leq C_1 (k+1)$ for some constant $C_1$.
    We also denote by $\tilde{\iota}$ the map induced on the quotient
    \begin{align*}
        \tilde{\iota}:(\CC[G]_{\leq k}\otimes B)/(\CC[G]_{\leq k}\otimes I)\to (\iota(\CC[G]_{\leq k})\otimes B)/(\iota(\CC[G]_{\leq k})\otimes I)
    \end{align*}
    which is contractive and $\|\tilde{\iota}^{-1}\|\leq C_1(k+1)$.
    Therefore, by defining the quotient map $\Tilde{\rho}_k:\iota(\CC[G]_{\leq k})\otimes_{\text{min}} B  \twoheadrightarrow (\iota(\CC[G]_{\leq k})\otimes_{\text{min}} B)/ (\iota(\CC[G]_{\leq k})\otimes_{\text{min}} I)$, we have for $f\in \CC[G]_{\leq k} \otimes B$,
    \begin{align*}
            \|\Tilde{\rho}(f)\| =&\|\tilde{\iota}^{-1}\Tilde{\rho}_k(\iota\otimes Id_B)(f)\|\\  
            \leq& C_1(k+1) \|\Tilde{\rho}_k(\iota\otimes Id_B)(f)\|\\
            {=}&   C_1(k+1)\|(\text{Id}\otimes \rho)(\iota \otimes Id_B)(f)\|\leq C_1(k+1)\|(\text{Id}\otimes \rho)(f)\|,
        \end{align*}
        where the equality in the last line follows from the nuclearity of $ M_N $.
    By applying this formula to $(f^*f)^n$, which is supported on $B_{2kn}$, we have
    \begin{align*}
         C_1(2kn+1)\|(Id \otimes \rho)(f)\|^{2n}= C_1(2kn+1)\|(Id\otimes \rho)(f^*f)^{2n}\|\geq\|\Tilde{\rho}((f^*f)^n)\|=\|\Tilde{\rho}(f)\|^{2n}.
    \end{align*}
    By taking the $2n$-th root on both side and let $n \to \infty$, we have $\|(\rho\otimes Id)(f)\|_{\text{min}}\geq\|\Tilde{\rho}(f)\|$.
    }
\end{rmk}

\begin{rmk}
\rm{
    Another natural operator valued analogue of Haagerup inequality can be stated as follows:
     there exist a positive integer $d$ and a constant $C$ such that for any $f \in \CC[G]\otimes B(\HH)$, we have
    \begin{align}\label{complete H-inequality}
        \|(\lambda\otimes 1)(f)\|_{B(\ell^2(G)\otimes \HH)} \leq C\left ( \left \|\sum_x (1+\ell(x))^{2d} f(x)^*f(x)\right\|^{\frac{1}{2}}+ \left\| \sum_x (1+\ell(x))^{2d} f(x)f(x)^*\right\|^{\frac{1}{2}} \right ).
    \end{align}
    This type of operator valued analogue (not exactly the same) has been exploited in \cite{christensen2021c} and proved for all groups with polynomial growth even with actions on a $C^*$-algebra. But one can directly show that \eqref{complete H-inequality} does not hold for the free group $\FF_2=\langle a,b \rangle$. Indeed, define $T_k:=\{g_1,g_2, \cdots , g_t\}\subset S_k$ to be the set of all reduced words with length $k$ starting from $a$ but not ending with $a^{-1}$. We have $\#T_k = t\geq 2^k$ for $k\geq 3$. We define $f \in \CC[G]_{2k}\otimes B(\HH)$ by
    \begin{align*}
        f(x)=\left\{
        \begin{array}{ll}
        E_{i,j} & (\text{if }x=g_ig_j)\\
        0 & (\text{otherwise}),
        \end{array}
        \right.
    \end{align*}
    where $E_{i,j}$ is the matrix unit $|e_i \rangle\langle e_j |$ for an orthonormal basis $\{e_i\}$ of $\HH$. Note that since $M_{k,k}(f)$ is a restriction of $(\lambda \otimes 1)(f)$,
    \begin{align*}
        \|M_{k,k}(f)\|\leq\|(\lambda\otimes 1)(f)\|_{B(\ell^2(G)\otimes \HH)}.
    \end{align*}
    Now by omitting rows and columns with only $0$-entries, we can regard $M:=M_{k,k}(f)$ as an operator from $\HH^{T_k}$ to itself, whose $(g_i,g_j)$-entry is $E_{i,j}$. Then $\|M\|= \#T_k \geq 2^k$ for $k \geq 3$. But
    \begin{align*}
        &\left \|\sum_x (1+\ell(x))^{2d} f(x)^*f(x)\right\|^{\frac{1}{2}}+ \left\| \sum_x (1+\ell(x))^{2d} f(x)f(x)^*\right\|^{\frac{1}{2}}  \\
        =& (1+2k)^{d}\left\{\|\sum_{i=1}^t\sum_{j=1}^t E_{j,j}\|^{\frac{1}{2}}+ \|\sum_{j=1}^t\sum_{i=1}^t E_{i,i}\|^{\frac{1}{2}} \right\}      \\
        =&2(1+2k)^d\sqrt{t}=2(1+2k)^d\sqrt{\#T_k}.
    \end{align*}
    Therefore no constants $d$ and $C$ satisfy \eqref{complete H-inequality} for all $f \in\CC[G]\otimes B(\HH)$.
    }
\end{rmk}

\begin{rmk}
\rm{
    One can also use the same strategy to estimate $ \|(\lambda\otimes 1) f\| $ by matrix of the form $ M_{i,j}(\cdot) $ with exactly $ i+j=k $ just like for the free groups. To see this, one simply need to decompose $y$ instead of $x$ in the proof of Lemma \ref{H-type inequality}. However, it turns out that in order to get the correct bound, one needs to divide the coefficients $f(y)$ of $f$ by the integers $d_{i,j}(y) := \#\{ (y_1,y_2)\in S_{i}\times S_{j}: y=y_1y_2 \}$. Namely, if we define $ \tilde{f}_{i,j}(y) = f(y)/d_{i,j} $, then we can show that
    $$ \|(\lambda\otimes 1) f\|\leq 2\cdot \# B_{1+2\delta}\sum_{i+j=k}\|M_{i,j}(\tilde{f}_{i,j})\|. $$
    However, as we do not know the completely bounded norm of the Schur multiplier given by $ S_{i}\times S_{j}\ni (y_1,y_2)\mapsto \delta_{\ell(y_1y_2),k}/d_{i,j}(y_1y_2) $, it is not clear whether one can actually show that
    $$ \|(\lambda\otimes 1) f\|\leq C\sum_{i+j=k}\|M_{i,j}(f)\|. $$ 
    }
\end{rmk}

\bibliographystyle{plain}
\bibliography{main}

\end{document}